\documentclass [12pt]{amsart}
\usepackage [utf8]{inputenc}
\usepackage[all]{xy}
\usepackage[english]{babel}
\usepackage{amssymb, enumerate}
\usepackage[lite, initials]{amsrefs}
\newtheorem{teorema}{Theorem}[section]
\newtheorem*{theorem*}{Main Theorem}
\newtheorem{lemma}[teorema]{Lemma}
\newtheorem{propos}[teorema]{Proposition}
\newtheorem{corol}[teorema]{Corollary}
\theoremstyle{definition}

\newtheorem{rem}{Remark}[section]
\newtheorem{defin}[teorema]{Definition}

\def\R{{\mathbb R}}
\def\N{{\mathbb N}}
\def\C{{\mathbb C}}

\def\Z{{\mathbb Z}}

\def\CP{{\mathbb{CP}}}
\def\H{{\mathbb H}}
\def\cU{{\mathcal U}}

\def\sfera{{\mathbb S}}

\def\Re{{\sf Re}}

\def\ord{{\rm ord}}

\renewcommand{\Im}{\mathsf{Im}}
\newcommand{\minor}[2]{\det\begin{pmatrix}z_{#1}&z_{#2}\\\tilde{z}_{#1}&\tilde{z}_{#2}\end{pmatrix}}

\title{Holomorphicity of slice-regular functions}
\author[S. Mongodi]{Samuele Mongodi}

\begin{document}

\maketitle

\tableofcontents

\section{Introduction}

The theory of slice-regular functions of a quaternionic variable poses itself as a possible generalization of the theory of holomorphic functions of a complex variable; the definition of slice-regular functions was given by Gentili and Struppa in \cites{GS1,GS2}.

Given $q\in\H$, we write $q=x+vy$ with $x,y\in\R$ and $v\in\H$ such that $v^2=-1$, i.e. an imaginary unity; for $\Omega\subset \H$ an axially symmetric domain, we call $f:U\to\H$ a left slice-regular function at $x+vy$ if it is real differentiable and
$$\frac{\partial f}{\partial x}(x+vy)+v\frac{\partial f}{\partial y}(x+vy)=0$$
for all imaginary units $v$. Some variations are possible, but all of them amount to saying that the function $f$ satisfies a sort of Cauchy-Riemann equations on each complex plain; such a condition implies a certain amount of symmetry, namely the map $v\mapsto f(x+vy)$ is linear affine.

This means that if we know the value of a slice-regular function $f$ at $x+vy$ for some imaginary unit $v$, then we can calculate it to all the other quaternions of the form $x+wy$, varying the imaginary unit $w$; given $\cU\subseteq\R^2$, we define $U\subseteq\H$ as the set of all quaternions $x+vy$ with $(x,y)\in\cU$, then any slice regular function $f:U\to\H$ is given by $f(x+vy)=\alpha(x,y)+v\beta(x,y)$, where $\alpha,\beta:\cU\to\H$ satisfy the Cauchy-Riemann equations
$$\left\{\begin{array}{l}\partial_x\alpha=\partial_y\beta\\\partial_y\alpha=-\partial_x\beta\;.\end{array}\right.$$

In \cite{GP1}, Ghiloni and Perotti built on the ideas of Fueter, Qian, Sce and Sommen  \cites{F,Q,S, So} and proposed a different viewpoint, involving an actual holomorphic function. Let $\C$ be the field of complex numbers, with imaginary unit $\iota$, and, for every imaginary unit $v$ of $\H$, consider the map $\rho_v:\C\to\H$ given by $\rho_v(x+\iota y)=x+vy$; this map can be extended to the tensor product $\H\otimes\C$ by setting
$$\rho_v(z\otimes q)=\rho_v(z)q$$
and extending it linearly.

So, given $\cU$ and $U$ as above, for every map $F:\cU\to\H\otimes\C$ such that $F(\overline{z})=\overline{F(z)}$ there exists a unique map $f:U\to\H$ such that the diagram
$$\xymatrix{\cU\ar[r]^{\!\!\!\!\! F}\ar[d]_{\rho_v}& \H\otimes\C\ar[d]^{\rho_v}\\U\ar[r]_{f}&\H}$$
commutes for every imaginary unit $v$. Now, $\H\otimes\C$ can be naturally identified with $\C^4$; then, $f$ is slice-regular if and only if $F:\cU\to\C^4$ is holomorphic.

It is reasonable to expect that the holomorphicity of $F$ implies the properties of slice-regular functions that mimic more closely the behaviour of classical holomorphic functions; however, this direction of investigation has not been pursued further.

One reason could be the fact that the map $\rho_v$, which allows us to pass from an element of $\C^4$ to an element of $\H$, somehow ruins the holomorphicity. So, for instance, there is no evident way of relating the zeros of $f$ to the zeros of $F$.

For an account of the theory of slice-regular functions and more generaly hypercomplex analysis, the interested reader can consult a variety of texts, for instance \cites{ACS, CSS1, CSS2, GSS}.

The aim of this work is to show how many results about slice-regular functions follow flawlessly from the properties of holomorphic functions.

\medskip

The first problem we consider is to investigate the zeros of $f$ in terms of the values of $F$: we show that the set of points $p$ in $\C^4$ such that $\rho_v(p)=0$ for some $v$ is a complex analytic set, fact related to the holomorphicity of $f^s$, the symmetrization of $f$.

We go on, studying the nature of the zeros of slice-regular functions and deducing the analogues of Rouché theorem and argument principle for slice-regular functions; the generalization to meromorphic functions is also straightforward.

The Cauchy formula for $f$ is deduced from the one for $F$; in general, if we have an integral kernel for an operator on holomorphic functions, we can extend it to an integral kernel for slice-regular functions, if some conditions are satisfied.

Finally, we study the relations between the supremum norm and the $L^2$ norm of $f$ and $F$; this allows us to obtain the maximum modulus principle for slice-regular functions and to use the previous observations on integral kernels to define a Bergman kernel for slice-regular functions.

We end the paper with a short section with some comments and examples on the possible extensions to the case of Clifford algebras, which will be the subject of a future paper \cite{Mo}.

\section{General setting}

Let $\H$ denote the algebra of quaternions and $\sfera$ the $2$-sphere of imaginary units, i.e.
$$\sfera=\{q\in\H\ :\ q^2=-1\}\;.$$
For a fixed $v\in\sfera$, we denote by $\C_v$ the (real) vector space generated by $1$ and $v$; obviously $\C_v$ is isomorphic to the usual field of complex numbers, through the map taking $x+vy$ to $x+\iota y$. Let $\rho_v:\C\to\C_v$ be the inverse of such map.

We extend the map $\rho_v$ to the tensor product $\H\otimes\C$ by setting $\rho_v(z\otimes q)=\rho_v(z)q$ and imposing linearity; in what follows, we will identify $\H\otimes\C$ and $\C^4$, using the base $\{1,i,j,k\}$ of $\H$.

For $q\in\H$, we write $q=q_0+q_1i+q_2j+q_3k$ with $q_0,\ldots, q_3\in\R$; we denote by $\underline{q}$ the vector part, i.e. $q-q_0$.

Let $\pi:\R^2\times \sfera\to\H$ be given by $\pi(x,y,v)=x+vy$; if $q\in\H$ is not real, then $\pi^{-1}(q)$ contains $2$ points, whereas, for $x\in\R$, $\pi^{-1}(x)=\{(x,0)\}\times\sfera$.

Let $\cU\subset\R^2$ be an open domain, invariant with respect to the map $(x,y)\mapsto (x,-y)$ and define $U=\pi_1(\cU\times\sfera)$ - the sets $U$ obtained this way are called spherically symmetric. We can identify $\R^2$ with $\C$ and interpret $\cU$ as an open domain in $\C$, invariant by conjugation.

\begin{defin}A (left) slice-regular function on $U$ is a map $f:U\to\H$ such that $f\circ \pi(x,y,v)=\alpha(x,y)+v\beta(x,y)$ where $\alpha,\beta:\cU\to\H$ are such that
$$\alpha(x,-y)=\alpha(x,y)\qquad \beta(x,-y)=-\beta(x,y)$$
$$\partial_x\alpha=\partial_y\beta \qquad \partial_y\alpha=-\partial_x\beta\;.$$
\end{defin}

\begin{lemma}\label{lemma_rappr} Given a (left) slice-regular function $f:U\to\H$, there exists a unique holomorphic map $F:\cU\to\C^4$ such that $F(\bar{z})=\overline{F(z)}$ and $f\circ \pi(x,y,v)=\rho_v(F(x+\iota y))$.\end{lemma}
\begin{proof} As $f$ is slice-regular, there exist functions $\alpha,\beta:\cU\to\H$ such that $f\circ \pi(x,y,v)=\alpha(x,y)+v\beta(x,y)$.

If we write $\alpha=\alpha_0+i\alpha_1+j\alpha_2+k\alpha_3$ and similarly $\beta=\beta_0+i\beta_1+j\beta_2+k\beta_3$, $\alpha_m, \beta_m:\cU\to\R$ are functions with the same symmetries and, in pairs, they respect Cauchy-Riemann equations, so we obtain that the functions $f_m(x+\iota y)=\alpha_m(x,y)+\iota\beta_m(x,y)$ for $m=0,1,2,3$ are holomorphic functions from $\cU\subseteq\C$ to $\C$, with the property
$$f_m(\bar{z})=\overline{f_m(z)}\quad \forall z\in \cU\;$$
and
$$f\circ\pi(x,y,v)=\rho_v(f_0(x+\iota y), f_1(x+\iota y), f_2(x+\iota y), f_3(x+\iota y))\;.$$

Thus, defining $F:\cU\to\C^4$ by $F=(f_0,f_1,f_2,f_3)$, the conclusion follows.\end{proof}

In what follows, we will denote by lower case letters the slice-regular functions and by upper case letters the corresponding holomorphic maps.

The study of the sets $\Pi^{-1}(q)$ is related to the link between the values taken by $f$ and the values taken by $F$.

For $q\in\H$, we define
$$Z(q)=\{z\in\C^4 \ :\ \exists v\in\sfera \textrm{ s.t. } \Pi(z,v)=q\}$$
and, for $q\in\H$ and $v\in\sfera$,
$$Z_v(q)=\{z\in\C^4\ :\ \Pi(z,v)=q\}$$
Also, for short, $Z(0)=Z$, $Z_v(0)=Z_v$. Obviously
$$Z=\bigcup_{v\in\sfera} Z_v$$
and
$$\Pi^{-1}(q)=\bigcup_{v\in\sfera}Z_v\times\{v\}\;.$$

\begin{rem}We will always consider left slice-regular functions in this paper; however, to obtain right slice-regular function it is enough to consider $\C\otimes\H$ and to extend $\rho_v$ as $\rho_v(q\otimes z)=q\rho_v(z)$.\end{rem}

\section{Geometry of the set $Z$}

We start by noticing the following

\begin{lemma} For $v\in\sfera$, $Z_v$ is a complex vector subspace of $\C^4$ of dimension $2$.\end{lemma}
\begin{proof} Let us write $z_m=z_m+\iota y_m$ for $m=0,\ldots, 3$ and $v=ai+bj+ck\in\sfera$. We compute
$$\Pi(z, v)=x_0+ay_0i+by_0j+cy_0k+x_1i-ay_1-by_1k+cy_1j+x_2j+$$
$$+ay_2k-by_2-cy_2i+x_3k-ay_3j+by_3i-cy_3=$$
$$=x_0-ay_1-by_2-cy_3+(ay_0+x_1-cy_2+by_3)i+(by_0+cy_1+x_2-ay_3)j+$$
$$+(cy_0-by_1+ay_2+x_3)k\;,$$
so, the set $Z_v$ is described by the four equations
$$\left\{\begin{array}{ll}x_0-ay_1-by_2-cy_3&=0\\
ay_0+x_1-cy_2+by_3&=0\\
by_0+cy_1+x_2-ay_3&=0\\
cy_0-by_1+ay_2+x_3&=0\end{array}\right.$$
which are clearly linear and independent. Therefore $Z_v$ is a real vector subspace of $\C^4$ and $\dim_\R Z_v=4$.
Moreover,
$$x_0-ay_1-by_2-cy_3+\iota(a(ay_0+x_1-cy_2+by_3)+b(by_0+cy_1+x_2-ay_3)+$$
$$+c(cy_0-by_1+ay_2+x_3)=$$
$$=x_0+\iota y_0+\iota ax_1-ay_1+\iota bx_2-by_2+\iota cx_3-cy_3=z_0+\iota (az_1+bz_2+cz_3)\;.$$
Hence we can describe $Z_v$ by the following system of (complex linear) equations
$$\left\{\begin{array}{ll}
-\iota z_0+az_1+bz_2+cz_3&=0\\
-az_0-\iota z_1+cz_2-bz_3&=0\\
-bz_0-cz_1-\iota z_2+az_3&=0\\
-cz_0+bz_1-az_2-\iota z_3&=0\end{array}\right.$$
whose rank on $\C$ is $2$; this means that $Z_v$ is the kernel of
\begin{equation}\label{eq_matrice}A_v=\begin{pmatrix}-\iota&a&b&c\\-a&-\iota&c&-b\\-b&-c&-\iota&a\\-c&b&-a&-\iota\end{pmatrix}\end{equation}
hence a complex vector subspace of dimension $2$.
\end{proof}

We consider the Grassmannian of $2$-planes in $\C^4$, $\mathrm{Gr}(4,2)$, so that we have a map $v\mapsto Z_v$ from $\sfera$ to $\mathrm{Gr}(4,2)$. We want to study the image of such map, so we embedd the Grassmannian into $\CP^5$ via the Pl\"ucker embedding.

Given $L\in\mathrm{Gr}(4,2)$, let $\{z,\tilde{z}\}$ be a base for $L$; we associate to $L$ the point in $\CP^5$ with homogeneous coordinates $[w_0,\ldots, w_5]$ given by
$$w_0=\minor{0}{1}\quad w_1=\minor{0}{2}\quad w_2=\minor{0}{3}$$
$$w_3=\minor{1}{2}\quad w_4=\minor{1}{3}\quad w_5=\minor{2}{3}\;.$$
The set of points $[w_0, \ldots, w_5]\in\CP^5$ obtained this way is the projective hypersurface
$$W=\{[w_0,\ldots, w_5]\in\CP^5\ : \ w_0w_5-w_1w_4+w_2w_3=0\}\;.$$

\begin{lemma}The image of the map $v\mapsto Z_v$ is a complex subspace of $\mathrm{Gr}(4,2)$.\end{lemma}
\begin{proof}We want to use the Pl\"ucker embedding, but in order to do this we would need a basis of $Z_v$; we note that, as $Z_v$ is the kernel of $A_v$, then $Z_v^\perp$ (which still belongs to $\mathrm{Gr}(4,2)$) is the range of $\overline{A}_v^t$, where $A_v$ is given by \ref{eq_matrice}.

Suppose, without loss of generality, that $a\neq \pm1$, so that $b^2+c^2\neq 0$, then $Z_v$ is the kernel of
$$\begin{pmatrix}-\iota&a&b&c\\-a&-\iota&c&-b\end{pmatrix}$$
and, consequently, $Z_v^\perp$ is the range of
$$\begin{pmatrix}\iota&-a\\a&\iota\\b&c\\c&-b\end{pmatrix}\;,$$
so that the point of $W$ corresponding to $Z_v^\perp$ is
$$[-1+a^2:ab+\iota c:ac-\iota b:ac-\iota b:-ab-\iota c:-b^2-c^2]\;.$$
It is easy to see that every such point is contained in
$$\mathcal{S}=\{[w_0:\ldots:w_5]\in W\ :\ w_0-w_5=w_1+w_4=w_2-w_3=0\}\;,$$
which is a complex projective non-degenerate conic in $\CP^5$, i.e. a copy of $\CP^1$.
The same result is obtained if one supposes $b\neq \pm1$ or $c\neq \pm1$.

To complete the proof, we have to show that $\mathcal{S}$ is indeed the image of the map $v\mapsto Z_v^\perp$.

We take $w\in \mathcal{S}$, i.e.
$$w=[w_0:w_1:w_2:w_2:-w_1:w_0]\qquad \textrm{with }w_0w_5-w_1w_4+w_2w_3=0$$
and we suppose $w_0\neq 0$ (which corresponds to $a\neq \pm 1$), so we can write $\zeta_1=w_1/w_0$, $\zeta_2=w_2/w_0$; then, taking $v=ai+bj+ck$ with
$$a=\frac{\Re \zeta_2}{\Im \zeta_1}$$
$$b=\frac{(\Im \zeta_2)^2-(\Re\zeta_1)^2}{\Im\zeta_2}$$
$$c=\frac{(\Re\zeta_2)^2-(\Im \zeta_1)^2}{\Im \zeta_1}$$
we have that the Pl\"ucker coordinates for $Z_v^\perp$ are exactly
$$[1:\zeta_1:\zeta_2:\zeta_2:-\zeta_1:1]=[w_0:w_1:w_2:w_2:-w_1:w_0]\;.$$
As $w$ was taken in $W$, we have $w_0w_5-w_1w_4+w_2w_3=0$, which, on $\mathcal{S}$, is equivalent to $w_0^2+w_1^2+w_2^2=0$; so $1+\zeta_1^2+\zeta_2^2=0$. From this relation it is a simple matter of computation to verify that $a^2+b^2+c^2=1$.
\end{proof}


We are now in the position for stating and proving the geometric description of $Z$.

\begin{teorema}\label{teo_zeroset}The set $Z$ is a complex hypersurface in $\C^4$, described by the equation
$$z_0^2+z_1^2+z_2^2+z_3^2=0\;.$$
In particular, it is the only critical level of the function $\Phi(z)=z_0^2+z_1^2+z_2^2+z_3^2$ and its only critical point is the origin.
\end{teorema}
\begin{proof} We consider the analytic set $\mathfrak{Z}\subseteq \C^4\times\CP^5$ defined by
$$\mathfrak{Z}=\{(z,w)\in\C^4\times \mathcal{S}\ :\ z_0w_3-z_1w_1+z_2w_0=z_0w_4-z_1w_2+z_3w_0=$$
$$=z_0w_5-z_2w_2+z_3w_1=z_1w_5-z_2w_4+z_3w_3=0\}\;.$$
One easily checks that $\mathfrak{Z}$ has complex dimension $3$.

Let $p_1:\C^4\times\CP^5$ be the projection on the first factor; then $Z=p_1(\mathfrak{Z})$. We introduce $Z^*=Z\setminus\{0\}$ and $\mathfrak{Z}^*=p_1^{-1}(Z^*)$ and we claim that
$$p_1\vert_{\mathfrak{Z}^*}:\mathfrak{Z}^*\to Z^*$$
is a biholomorphism. Indeed, given $z\in Z^*$, a pair $(z,w)$ belongs to $\mathfrak{Z}^*$ if and only if $w\in\ker B$, with
$$B=\begin{pmatrix}z_2&-z_1&0&z_0&0&0\\z_3&0&-z_1&0&z_0&0\\
0&z_3&-z_2&0&0&z_0\\0&0&0&z_3&-z_2&z_1\\1 &0 &0&0&0&-1\\0&1&0&0&1&0\\0&0&1&-1&0&0\end{pmatrix}\;.$$
Computing the minors of $B$, one has that if $\mathrm{rk} B<5$ then $z=0$; so, for $z\in Z^*$, there exists a unique $w\in \CP^5$ such that $p_1(z,w)=z$. It follows that $p_1$ is a biholomorphism between $\mathfrak{Z}^*$ and $Z^*$; therefore, $Z$ is a complex analytic set, by the Remmert-Stein theorem. Moreover, $\dim_\C Z=3$, so there exists a global equation for it in $\C^4$.

If we explicitly solve $Bw=0$ as a function of $z$ and we require that the solution lies on $W$, we obtain the following equation
$$z_0^2+z_1^2+z_2^2+z_3^2=0\;.$$
By irreducibility, we conclude that $Z=\{z\in\C^4\ :\  z_0^2+z_1^2+z_2^2+z_3^2=0\}$; if we consider the holomorphic function $\Phi:\C^4\to\C$, $\Phi(z)=z_0^2+z_1^2+z_2^2+z_3^2$, then $Z$ is the preimage of $0$ and the only critical level set. Moreover, the only singular point of $Z$ is the origin. \end{proof}

\begin{rem}We note that the previous construction gives a diffeomorphism between $\sfera$ and $\mathcal{S}\cong\CP^1$, thus inducing a complex structure on $\sfera$; moreover, $\mathcal{S}$ being a conic inside $W$, this map can be seen as taking values in the Grassmannian of $2$-planes in $\C^4$, so it gives a (holomorphic) fiber bundle on $\sfera$. The set $\mathfrak{Z}$ is the total space of such a bundle and $Z$ is the space obtained by contracting the zero section to a point. \end{rem}

\begin{rem}The sets $\mathfrak{Z}$ and $Z$, being described by equations with real coefficients, are invariant under conjugation of all variables. This reflects the fact that $x+vy=x+(-v)(-y)$.\end{rem}

We conclude this section with some observations on the properties of $Z$ and $Z(q)$.

\begin{corol}\label{corol_prop}\begin{enumerate}
\item The set $Z(q)$ is the zero locus of the function $\Phi_q(z)=(z_0-q_0)^2+(z_1-q_1)^2+(z_2-q_2)^2+(z_3-q_3)^2\;.$
\item The fundamental group of $\C^4\setminus Z(q)$ is $\Z$, for every $q\in\H$.
\item If $f:U\to\H$ is a slice-regular function and $F:\mathcal{U}\to\C^4$ is the map given by Lemma \ref{lemma_rappr}, then $F(\mathcal{U})\cap Z(q)$ is discrete for every $q\in\H$.
\end{enumerate}
\end{corol}
\begin{proof} Statement $(1)$ is obvious. As for statement $(2)$, we restrict our attention to the case $q=0$, the general case being equivalent up to translation.

The function $\Phi$ has non-vanishing gradient on $\C^4\setminus Z$, hence it induces a fibration, by Ehresmann theorem (see Theorem 9.3 and Remark 9.4 in \cite{Vo}); the generic fiber is
$$\{z\in\C^4\ :\ z_0^2+z_1^2+z_2^2+z_3^2=c\}$$
for $c\in\C^*$. It is well known that such a set is diffeomorphic to the total space of the tangent bundle to $\sfera^3$, hence simply connected. From the long exact sequence for the homotopy of a fibration, we get
$$\pi_1(\C^4\setminus Z)\cong\pi_1(\C^*)\cong\Z\;.$$

To prove $(3)$, again we only consider the case $q=0$. By construction, the open Riemann surface $F(\mathcal{U})$ intersects $\R^4\subset\C^4$ in a set with positive linear measure; however, $Z$ intersects $\R^4\subset\C^4$ only in one point; this means that $F(\mathcal{U})$ is not contained in $Z$, so their intersection must be an analytic set of dimension $4-3-1=0$, i.e. a discrete set (possibly empty).

\end{proof}

\section{Zeros of slice-regular functions}

Let $f:U\to\H$ be a slice-regular function and $F:\cU\to\C^4$ be the holomorphic map obtained from Lemma \ref{lemma_rappr}; suppose $q\in U$ is a point where $f(q)=0$ and write $q=x+vy$. Then $F(x+\iota y)$ and $F(x-\iota y)$ belong to $Z\subset\C^4$.

If $y=0$, then we have a real $x$ such that $F(x)\in Z$, but $F(x)\in\R^4\subset\C^4$ and $Z\cap\R^4=\{(0,0,0,0)\}$, so this implies that $F(x)=(0,0,0,0)$. Similarly, if there are two points $q, q'$, with $q=x+vy$ and $q'=x+v'y$, where $f(q)=f(q')=0$, then $F(x+\iota y)=F(x-\iota y)=(0,0,0,0)$: indeed, let $F(x+\iota y)=z\in\C^4$, then there should be $w, w'\in\mathcal{S}$, with $w'\neq \bar{w}$, such that $p_1(z,w)=p_1(z,w')=0$, but this is possible only if $z=(0,0,0,0)$.

Combining these observations with Corollary \ref{corol_prop}, we immediately obtain that the zeros of $f$, generic, real or spherical, are a discrete set and no two isolated zeros are on the same sphere.

From this geometric perspective, an immediate consequence of Theorem \ref{teo_zeroset} is the following.
\begin{propos}\label{cor_equiv}Let $f:U\to\H$ be a slice-regular function; a point $x+vy\in U$ is a zero of $f$ if and only if $x+\iota y$ is a zero for the holomorphic function $\Phi\circ F$. Moreover,
\begin{enumerate}
\item if $x+vy$ is an isolated non-real zero of multiplicity $k$ for $f$, then $x+\iota y$ and $x-\iota y$ are zeros of multiplicity $k$ for $\Phi\circ F$;
\item if $[x+vy]$ is a spherical zero of multiplicity $k$ for $f$, then $x+\iota y$ and $x-\iota y$ are zeros of multiplicity $2k$ for $\Phi\circ F$
\item if $x$ is an isolated real zero of multiplicity $k$ for $f$, then $x$ is a zero of multiplicity $2k$ for $\Phi\circ F$. \end{enumerate}
\end{propos}
\begin{proof} The only part needing a proof is the one about multiplicities, which follows easily if one notices that $(0,0,0,0)$ is a singular point for $\Phi$, where $\nabla \Phi$ vanishes, but $\nabla^2\Phi$ is non-degenerate: both real zeros and spherical zeros correspond to points $z\in\cU$ where $F(z)=(0,0,0,0)$.

So, for an isolated non-real zero, the multiplicity of the intersection between $F(\cU)$ and $Z$ is the multiplicity of zero of $\Phi\circ F$, whereas for a real or spherical zero, the multiplicity gains a factor $2$ because $(0,0,0,0)$ is a singular point of $Z$, namely, a double point.\end{proof}

The previous result means that we can use all the techniques, which we employ to study the zeros of holomorphic functions, to understand the zeros of slice-regular functions. But there is more to it: such zeros are actually fully characterized as intersection points between two complex analytic sets in $\C^4$ (a Riemann surface and the quadric cone).

Some of the results we present in the following were already obtained by \cite{V}, however, we hope that this presentation can shed new light on the nature of such results.

\begin{rem}The function $\Phi\circ F$ is linked to the sphericization of $f$, namely
$$f^s\circ\pi(x,y,v)=\rho_v\circ\Phi\circ F(x+\iota y)\;.$$
\end{rem}

\begin{teorema}[Counting zeros - I] \label{teo_contazeri1}Let $f$ be a slice-regular function defined on a neighbourhood of the closure of $U$, suppose that $f$ does not vanish on $bU$. If $f$ has $k$ isolated (real or non-real) zeros and $m$ spherical zeros in $U$ (counted with multiplicities), then
$$\frac{1}{2\pi\iota}\int_{b\cU}\frac{(\Phi\circ F)'(z)}{(\Phi\circ F)(z)}dz=2k+4m\;,$$
where $z=x+\iota y$ and $bU$, $b\cU$ are supposed to be regular enough for the integration to make sense.
\end{teorema}
\begin{proof}It is enough to apply the known result for holomorphic function and to keep in mind Corollary \ref{cor_equiv}.\end{proof}

For a slice-regular function as in the hypotheses of the previous result, we call $k+2m$ the \emph{weighted number of zeros} in $U$.

 We can specialize this observation to entire function and recover the classical theory of distribution of zeores.

In particular,given $f:\H\to\H$, in accordance with Theorem \ref{teo_contazeri1}, we define
$$n_f(t)=k_f(t)+2m_f(t)$$
where $k_f(t)$ is the number of isolated zeros of $f$ in the ball $B(0,t)=\{q\in\H\ :\ |q|<t\}$ and $m_f(t)$ is the number of spherical zeros of $f$ in the same ball. This is the same counting function defined in \cite{CSS1}.

Then, recalling the classical Jensen formula, we obtain that
$$\int_0^R\frac{n_f(t)}{t}dt=\dfrac{1}{4\pi}\int_0^{2\pi}\log\left|\Phi\circ F(Re^\iota \theta)\right|d\theta - \log\left|\Phi\circ F(0)\right|\;,$$
which is the same result obtained in \cite{CSS1}.

The following version of Rouché theorem is also readily obtained by the standard statement of one complex variable.

\begin{teorema}[Rouché theorem for slice-regular functions]
Suppose that $f,g$ are slice-regular functions defined on a neighbourhood of the closure of $U$. Suppose that, for each $z\in b\cU$,
$$|(\Phi\circ F)-(\Phi\circ G)|<|(\Phi\circ F)| + |(\Phi\circ G)|\;.$$
Then $f$ and $g$ have the same weighted number of zeros in $U$.
\end{teorema}
\begin{proof} It is enough to apply Rouché theorem \cite{C}*{Theorem 3.8 -- p. 125} to the functions $\Phi\circ F$ and $\Phi\circ G$.
\end{proof}

It is interesting to remark that Corollary \ref{corol_prop} ensures that we are not losing any topological information about the winding number by using the function $\Phi\circ F$ in place of the function $F$ (or the function $f$).

\medskip

So far, the integration was carried out on the boundary of sets which are symmetric with respect to the real axis; the next theorem considers the case where this condition no longer holds.

\begin{teorema}[Counting zeros - II]\label{teo_conteggio2}
Let $f:U\to\H$ be a slice-regular function; consider $\Omega\subset \cU$ and let
\begin{itemize}
\item $k_0$ be the number of isolated non-real zeros $q=\pi(x,y,v)$ of $f$ with $z=x+\iota y\in\Omega$ and $\bar{z}\in\Omega$
\item $m_0$ be the number of spherical zeros $[q]=\pi(\{(x,y)\}\times \sfera)$ of $f$ with $z=x+\iota y\in\Omega$ and $\bar{z}\in\Omega$
\item $k_1$ be the number of isolated non-real zeros $q=\pi(x,y,v)$ of $f$ with $z=x+\iota y\in\Omega$ and $\bar{z}\not\in\Omega$
\item $m_1$ be the number of spherical zeros $[q]=\pi(\{(x,y)\}\times \sfera)$ of $f$ with $z=x+\iota y\in\Omega$ and $\bar{z}\not\in\Omega$
\item $r$ be the number of real zeros $x$ of $f$, with $x\in\Omega$.
\end{itemize}
Then we have
$$\frac{1}{2\pi\iota}\int_{b\Omega}\frac{(\Phi\circ F)'(z)}{(\Phi\circ F)(z)}dz=2k_0+k_1+2r+4m_0+2m_1\;.$$
\end{teorema}

Cunningly varying the set $\Omega$, we can include or exclude every species of zeros in the counting and thus obtain various linear equations, in order to calculate the number of isolated non-real, spherical and real zeros of $f$ in a given spherically symmetric open set.

Finally, the classical Hurwitz theorem easily implies its version for slice-regular functions.

\begin{teorema}Let $\{f_n\}_{n\in\N}$ be a sequence of non-vanishing slice-regular functions $f_n:U\to\H$ which converges uniformly to $f:U\to\H$. Then $f$ is either constantly zero or non-vanishing.\end{teorema}
\begin{proof} It is enough to apply the classical Hurwitz theorem to the sequence of functions $\Phi\circ F_n$, once one checks that their sequence converges to $\Phi\circ F$; the latter is easy: if $g:U\to\H$ is slice-regular and $G$ is the holomorphic map given by Lemma \ref{lemma_rappr}, then, once we write $g\circ \pi(x,y,v)=\alpha(x,y)+v\beta(x,y)$, then
$$|G(x+\iota y)|^2=|\alpha(x,y)|^2+|\beta(x,y)|^2\leq\max_{v\in\sfera}|g\circ\pi(x,y,v)|^2\;,$$
so if $\|f_n-f\|_{\infty}\to0$, then $\|F_n-F\|_{\infty}\to 0$ and, as $\Phi$ is also holomorphic, the conclusion follows.\end{proof}

\section{$\star$-products and poles}

Let $f,g:U\to\H$ be slice-regular functions; it is well known that their punctual product is not slice-regular and, to overcome this difficulty, another product, called $\star$-product, is defined. If we write
$$f\circ \pi(x,y,v)=\alpha(x,y)+v\beta(x,y) \qquad g\circ \pi(x,y,v)=\gamma(x,y)+v\delta(x,y)$$
then
\begin{equation}\label{eq_starprod}(f\star g)\circ\pi(x,y,v)=(\alpha\gamma-\beta\delta)(x,y)+v(\alpha\delta+\beta\gamma)(x,y)\;.\end{equation}
We define a ($\R$-homogeneous) product on $\C^4$ by identifying $\C^4$ with $\H\otimes_\R\C$ and we denote such product by $\star$.

\begin{lemma} \label{lemma_omom}If $f,g,h:U\to\H$ are slice-regular functions such that $f\star g=h$, then $H=F\star G$ and $\Phi\circ H=(\Phi\circ F)(\Phi\circ G)$.\end{lemma}
\begin{proof}
The $\star$-product on $\C^4$ works as follows: it is $\R$-homogeneous and
$$\iota\star\iota=i\star i=j\star j=k\star k=-1$$
$$i\star j= -j\star i = k\qquad j\star k= -k\star j= -i \qquad k\star i= - i\star k= j$$
where $\{1,i,j,k\}$ is the standard ($\C$-)base for $\C^4$ and $1$ and $\iota$ commute with everything.

So
$$F\star G=(\Re F)\star(\Re G) - (\Im F)\star(\Im G) + \iota ((\Re F)\star(\Im G) + (\Im F)\star(\Re G))$$
where the $\star$-product now is between vectors with real entries, so it is the usual quaternionic product. This is the same formula as in \eqref{eq_starprod}, which shows that $H=F\star G$.

Now, the equality $|ab|^2=|a|^2|b|^2$ for $a,b\in\H$, being an algebraic identity, holds also in $\C^4$ with the $\star$-product and the function $\Phi$ in place of the squared norm:
$$\Phi(A\star B)=\Phi(A)\Phi(B)\qquad \textrm{for } A,B\in\C^4\;.$$
This concludes the proof.
\end{proof}

An easy consequence of this lemma is the following proposition.

\begin{propos}\label{propos_prodzero}Let $f,g:U\to\H$ be slice-regular functions; if either $f$ or $g$ have a zero on the sphere $[q]$, then also $f\star g$ has a zero on that sphere.\end{propos}
\begin{proof}Let $h=f\star g$. First of all, we note that
$$\Phi\circ H=(\Phi\circ F)(\Phi\circ G)=(\Phi\circ G)(\Phi\circ F)\;,$$
so we can suppose, without loss of generality, that the function vanishing on a point of the sphere $[q]$ is $f$.

Let $q=x+vy$, then $\Phi\circ F(x+\iota y)=0$, which obviously implies that $\Phi\circ H(x+\iota y)=0$, so there exists at least one $v'\in\sfera$ such that $h(x+v'y)=0$.
\end{proof}

\begin{defin}Let $f:U\to\H$ be a slice-regular function; for $q\in U$, $q=x+vy$, we define the order (of zero) of $f$ on the sphere $[q]$ as
$$\ord_f([q])=\left\{\begin{matrix}\ord_{\Phi\circ F}(x+\iota y) & \textrm{if }y\neq 0\\ & \\\ord_{\Phi\circ F}(x)/2&\textrm{if }y=0\end{matrix}\right.$$
where the order of zero of $\Phi\circ F$ is defined as usual for a holomorphic function.
\end{defin}

\begin{propos}Let $f,g:U\to\H$ be slice-regular functions and $q\in U$. Then
$$\ord_{f\star g}([q])=\ord_{f}([q])\ord_g([q])\;.$$
\end{propos}
\begin{proof}
The identity follows from the proof of Proposition \ref{propos_prodzero}.
\end{proof}

A slice-meromorphic function is of the form $f^{-\star}\star g$, where $f,g:U\to\H$ are slice-regular functions.

\begin{propos}Let $f:U\to\H$ be a slice-regular function and define $\mathcal{V}=\{z\in\cU\ :\ F(z)\neq 0\}$ and $V=\pi(\mathcal{V}\times\sfera)$. We set $h=f^{-\star}$, so that $h:V\to\H$ is slice-regular.

Then, $\Phi\circ H$ extends to a meromorphic function on $\cU$; moreover,  if $\ord_f([q])=k$, then $\Phi\circ H$ has a pole of order $k$ in $x+\iota y$, with $q=x+vy$.
\end{propos}
\begin{proof} As noted before, $\cU\setminus\mathcal{V}$ is a discrete set, so $\Phi\circ H$ is meromorphic on $U$; moreover, by Lemma \ref{lemma_omom}, we have that
$$\Phi\circ H=\frac{1}{\Phi\circ F}$$
so, the poles of $\Phi\circ H$ are the zeros of $\Phi\circ F$ and the last part of the statement follows.
\end{proof}

So, we can extend the definition of order to zeros and poles of a meromorphic function: if $h=f^{-\star}\star g$, then we set
$$\ord_h([q])=\ord_g([q])-\ord_f([q])\;.$$
However, order zero does not mean a removable singularity. For example, the function
$$f(q)=\frac{1}{q^2+1}(q^2+q(i-j)-k)$$
has $\ord_f([i])=0$, because it has a double isolated zero on the sphere, which is also a pole of order $2$, but it does not extend to a slice-regular function on $\H$. On the other hand, the corresponding holomorphic function $\Phi\circ F$ is the constant function $1$.

Theorem \ref{teo_conteggio2} can be extended to meromorphic functions in the usual way.

\begin{teorema}Let $f:U\to\H$ be a slice-meromorphic function and let $V\subset U$ be the maximal axially symmetric open set such that $f$ is slice-regular on $V$. Consider $\Omega\subset\cU$ such that $f$ does not have zeros or poles on $b\Omega$. The quantities $k_0, k_1, m_0, m_1, r$ are as in Theorem \ref{teo_conteggio2}, moreover denote by
\begin{itemize}
\item $p_0$ the number of poles $[q]=\pi(\{(x,y)\}\times\sfera)$ of $f$ with $z=x+\iota y\in\Omega$ and $\bar{z}\in\Omega$
\item $p_1$ the number of poles $[q]=\pi(\{(x,y)\}\times\sfera)$ of $f$ with $z=x+\iota y\in\Omega$ and $\bar{z}\not\in\Omega$.
\end{itemize}
Then we have
$$\frac{1}{2\pi\iota}\int_{b\Omega}\frac{(\Phi\circ F)'(z)}{(\Phi\circ F)(z)}dz=2k_0+k_1+2r+4m_0+2m_1-2p_0-p_1\;.$$
\end{teorema}

We have so far characterized the poles of $f$ only in terms of poles of $\Phi\circ F$; however, given that a pole is always spherical, the behaviour of $F$ around a pole is also quite easy.

\begin{propos} Let $f:U\to\H$ be a slice-meromorphic and $V\subset U$ the maximal open set such that $f$ is slice-regular on $V$. If $w\in\cU\setminus\mathcal{V}$, then
$$\lim_{z\to w} |f_m(z)|=\infty\qquad m=0,1,2,3\;,$$
where $F=(f_0,f_1,f_2,f_3)$.\end{propos}

So, the poles of $f$ correspond to points where all the components of $F$ have a pole, in accordance with the fact that they are spherical zeros of the denominator, i.e. zeros where all the components of the holomorphic map go to zero.

\section{Integral kernels}

In the two previous sections, we employed the holomorphicity of the function $\Phi\circ F$ in order to obtain some information about the zeros of the slice-regular function $f$. However, as we noted earlier, studying the function $\Phi\circ F$ is equivalent to studying the function $f^s$, which amounts to a loss of information, because different slice-regular functions can give rise to the same symmatrization.

If we want to encode all the information about $f$, we have to look at the holomorphic map $F:\cU\to\C^4$. This presents a greater complexity, but allows us to prove more.

As an example, the classical Cauchy formula for vector-valued holomorphic functions of one variable implies that
\begin{equation}\label{eq_cauchy1}\frac{1}{2\pi\iota}\int_{b\cU}\dfrac{1}{\zeta-z}F(\zeta)d\zeta=F(z)\end{equation}
for all $z\in\cU$. From this identity, we immediately obtain a way to determine $f(q)$ in terms of the values of $F$ on the boundary of $\cU$; however, we would prefer an integral formula (like the Cauchy formula) expressing $f(q)$ in terms of the values of $f$ on some $bU\cap\C_v$, $v\in\sfera$.

\begin{lemma}\label{lmm_commute}For every $p(z), q(z)\in\R[z]$, we have
$$\rho_v(p(z)q^{-1}(z))=p(\rho_v(z))q^{-1}(\rho_v(z))\;.$$
\end{lemma}
\begin{proof}
It is quite obvious that
$$\rho_v(z^k)=(\rho_v(z))^k$$
and also, for $a\in\R$,
$$\rho_v(az)=a\rho_v(z)\;,$$
therefore $p(\rho_v(z))=\rho_v(p(z))$.

We note that $|\rho_v(z)|=|z|$ (where the first is computed in $\H$ and the second in $\C$),so, given a converging power series
$$\sum_{n=0}^\infty z^na_n$$
we have that
$$\rho_v\left(\sum_{n=0}^\infty z^na_n\right)=\sum_{n=0}^\infty\rho_v(z)^na_n\;.$$

Moreover, if $q(z)$ has real coefficients, we have
$$p(\rho_v(z))q^{-1}(\rho_v(z))=q^{-1}(\rho_v(z))p(\rho_v(z))$$
and also
$$p(\rho_v(z))q^{-1}(\rho_v(z))=p(\rho_v(z))\rho_v(q^{-1}(z))\;.$$
Finally, $p(z)q^{-1}(z)$ can be expressed as a converging power series around some point $z_0$ and we can assume it to be the origin, so
$$\rho_v(p(z)q^{-1}(z))=p(\rho_v(z))q^{-1}(\rho_v(z))\;.$$
\end{proof}

In the formula \eqref{eq_cauchy1}, the complex variable is $z$, $\zeta$ being the integration variable, so we would like to write everything in terms of power series of $z$ with real coefficients. As we already know that the components of $F$ have this property, the integral in \eqref{eq_cauchy1} has to produce functions whose power series have real coefficients, so we turn our attention to the part depending explicitly on $z$:
$$\dfrac{1}{\zeta-z}=-\dfrac{1}{z-\zeta}=-\dfrac{z-\overline{\zeta}}{z^2-2\Re\zeta z + |\zeta|^2}\;.$$
Hence, by Lemma \ref{lmm_commute}, we write
\begin{eqnarray*}f(q)&=&\dfrac{1}{2\pi \iota}\int_{b\cU}\dfrac{\overline{\zeta}-q}{q^2-2\Re\zeta q + |\zeta|^2}f_0(\zeta)d\zeta\\
&+&\dfrac{1}{2\pi \iota}\int_{b\cU}\dfrac{\overline{\zeta}-q}{q^2-2\Re\zeta q + |\zeta|^2}f_1(\zeta)d\zeta i\\
& +& \dfrac{1}{2\pi \iota}\int_{b\cU}\dfrac{\overline{\zeta}-q}{q^2-2\Re\zeta q + |\zeta|^2}f_2(\zeta)d\zeta j\\
 &+ &\dfrac{1}{2\pi \iota}\int_{b\cU}\dfrac{\overline{\zeta}-q}{q^2-2\Re\zeta q + |\zeta|^2}f_3(\zeta)d\zeta k\end{eqnarray*}
 and, reordering the terms in each integral, we can write
$$f(q)=\dfrac{1}{2\pi \iota}\int_{b\cU}\dfrac{\overline{\zeta}-q}{q^2-2\Re\zeta q + |\zeta|^2}d\zeta (f_0(\zeta)+f_1(\zeta)i+f_2(\zeta)j+f_3(\zeta)k)\;.$$
Now, as the components of $F$ have power series expansions with real coefficients, the integral itself cannot depend on $\iota$; indeed, if one rearranges the elements correctly, one obtains an expression that doesn't depend on $\iota$ even if the latter is not assumed to commute with the three quaternionic imaginary units. The correct order involves also the choice of the side on which we multiply by the inverse of the denominator: this can be worked out observing that it is should be right slice-regular in the variable $\zeta$, therefore the quotient becomes a multiplication on the left by the inverse of the denominator.

So, we can substitute $\zeta$ with a quaternionic variable $s$ varying on $bU\cap \C_v$ for some $v\in\sfera$. We have proven the following.

\begin{teorema}Let $f:U\to\H$ be a slice-regular function which extends continuously to the boundary, then, for every $q\in U$,
$$f(q)=\dfrac{1}{2\pi}\int_{bU\cap\C_v}S^{-1}_L(q,s)\dfrac{ds}{v}f(s)\;,$$
where $S^{-1}_L(q,s)=-(q^2-2\Re s q+|s|^2)^{-1}(q-\overline{s})$.\end{teorema}

With analogous arguments, we can obtain many related results, like the integral formulas for the derivatives and the Cauchy estimates for them.

\medskip

In general, let us consider a function $K:\cU\times b\cU\to\C$ of the form
$$K(z,w)=\sum_{n=0}^\infty\sum_{m=-\infty}^\infty a_{nm}z^nw^m$$
with $a_{nm}\in\R$.

It defines an integral operator on holomorphic functions on $\cU$ by
$$Kh(z)=\int_{b\cU}K(z,w)h(w)dw\;.$$
This integral operator can be extended to slice-regular functions in two meaningful ways:
\begin{enumerate}
\item by formally replacing $z$ and $w$ with quaternionic variables in $K$ and writing
$$K_\H f(q)=\int_{U\cap\C_v} K(q,s)\dfrac{1}{v}dsf(s)$$
where $K(q,s)=\sum_{n=0}^\infty\sum_{m=-\infty}^\infty a_{nm}q^ns^m$;
\item by applying the operator $K$ to the four components of $F$.
\end{enumerate}

\begin{rem}\label{rem_integkernel}The computations and considerations carried out for the Cauchy formula apply to such a general case and make us conclude that these two extensions give us the same integral operator on slice-regular functions, namely
$$\rho_v\left(Kf_0(x+\iota y)\right)+\rho_v\left(Kf_1(x+\iota y)\right)i+\rho_v\left(Kf_2(x+\iota y)\right)j+$$
$$+\rho_v\left(Kf_3(x+\iota y)\right)k=K_\H f(\pi(x,y,v))\;.$$
\end{rem}

Moreover, the same can be said when integrating on the whole set.

\begin{rem}\label{rem_integkernel2}If we define
$$Kh(z)=\int_{\cU} K(z,w)h(w)d\mu$$
where $\mu$ is the Lebesgue measure on $\cU$, we reach the same conclusion as in the previous remark. The operator on slice-regular functions will be then defined as
$$K_\H f(q)=\int_{U\cap\C_v}K(q,s)f(s)d\mu\;.$$
\end{rem}

\section{Norms}
As a last application, we look at the relations between the $L^\infty$ and $L^2$ norms of $f$ and of $F$.

\begin{propos}Let $f:U\to\H$ be a slice-regular function. Suppose there exists $q\in U$, $q=s+v_0t$, such that $|f|$ has a local maximum at $q$; then $|F|$ has a local maximum in $s+\iota t$.\end{propos}
\begin{proof}Let $f\circ\pi(x,y,v)=\alpha(x,y)+v\beta(x,y)$; we note that
\begin{equation}\label{eq_max}\max_{v\in\sfera}|f\circ \pi(x,y,v)|=|\alpha(x,y)|+|\beta(x,y)|\end{equation}
whereas
$$|F(x+\iota y)|=\sqrt{|\alpha(x,y)|^2+|\beta(x,y)|^2}\;,$$
so, in general these two values do not coincide and the latter is the lower. However, let us suppose that $V$ is a neighbourhood of $q$ in $U$ such that for every $p\in V$, $|f(p)|\leq|f(q)|$ and, without loss of generality, let us assume that $\alpha(s,t)=0$.

By continuity, if $(x,y)$ is close enough to $(s,t)$, then there is a $v\in \sfera$ close enough to $v_0$ which realizes the maximum in \eqref{eq_max}. So, we can suppose, up to shrinking $V$, that for every $p\in V$, the point on $[p]$ that realizes the maximum in \eqref{eq_max} is in $V$.

Then, for every $x+\iota y$ such that $\pi(\{(x,y)\}\times\sfera)\cap V\neq \empty$, let $p=x+vy$ where the maximum in \eqref{eq_max} is attained; we have
$$|F(x+\iota y)|\leq|f(p)|=|\alpha(x,y)|+|\beta(x,y)|\leq|\alpha(s,t)|+|\beta(s,t)|=$$
$$=|\beta(s,t)|=\sqrt{|\alpha(s,t)|^2+|\beta(s,t)|^2}=|F(s+\iota y)|\;.$$

This concludes the proof.\end{proof}

\begin{teorema}Let $f:U\to\H$ be a slice-regular function such that $|f|:U\to\R$ has a maximum in a (interior) point $q\in U$. Then $f$ is constant.\end{teorema}
\begin{proof}By the previous proposition, if the point $q=x+v y$ is a maximum point for $|f|$, then $x+\iota y\in\cU$ is a maximum point for $|F|$. The latter is a vector-valued holomorphic map, so its norm attains the maximum on the boundary, unless the map itself is constant; but, if $F$ is constant, so is $f$, which proves our result.
\end{proof}

In general, $|f(q)|$ and $|F(x+\iota y)|$ are different, but we have that
\begin{equation}\label{eq_norms}\min_{v\in\sfera}|f(x+vy)|\leq |F(x+\iota y)|\leq \max_{v\in\sfera}|f(x+vy)|\;.\end{equation}
So, a number of results that hold for holomorphic functions can be adapted to slice-regular functions.

\begin{propos}Let $f:\H\to\H$ be a slice-regular function such that $|f(q)|\sim |q|^n$ when $|q|\to\infty$. Then $f$ is a polynomial of degree at most $n$.\end{propos}
\begin{proof}From \eqref{eq_norms}, we conclude that also the components of $F$ are entire holomorphic functions that grow as a power of $|z|$; by a standard result of one complex variable, all the components of $F$ are polynomials of degree at most $n$, which easily implies the conclusion.
\end{proof}

As we have just seen, the norm of the map $F$ is controlled by the norm of $f$ on a sphere; if we look at the $L^2$ norm of these functions, however, we find a closer connection.

\begin{lemma}Let $f:U\to\H$ be a slice-regular function. For $q\in U$, $q=x+vy$, we have
$$\int_{\sfera}|f(x+v'y)|^2d\sigma(v')=4\pi|F(x+\iota y)|^2\;,$$
where $d\sigma$ is the area measure on the unit sphere $\sfera$.
\end{lemma}\begin{proof}We write
$$f\circ\pi(x,y,v')=\alpha(x,y)+v'\beta(x,y)$$
so
$$|f(x+v'y)|^2=|\alpha(x,y)|^2+|\beta(x,y)|^2+\alpha(x,y)\overline{\beta(x,y)v'}+v'\beta(x,y)\overline{\alpha(x,y)}\;.$$
Now, $x,y$ are fixed, so we set $\alpha(x,y)=a$ and $\beta(x,y)=b$. The functions $a\overline{bv'}$ and $v'b\overline{a}$, as functions of $v'$, are odd, so their integral on $\sfera$ vanishes. Hence
$$\int_{\sfera}|f(x+v'y)|^2d\sigma(v')=\int_{\sfera}(|a|^2+|b|^2)d\sigma(v')=4\pi(|a|^2+|b|^2)$$
and the last quantity is exactly $4\pi|F(x+\iota y)|^2$.
\end{proof}

From this lemma one immediately obtains the following result.

\begin{corol}Let $f:U\to\H$ be a slice-regular function, then
$$\int_{U}|f(q)|^2dx_0dx_1dx_2dx_3=4\pi\int_{\cU}y^2|F(x+\iota y)|^2dxdy$$
$$\int_{U\cap \C_v}|f(x+vy)|^2dxdy=\int_{\cU}|F(x+\iota y)|^2dxdy$$
for any $v\in\sfera$.\end{corol}

So, the Hilbert space of slice-regular functions on $U$ with the $L^2$-norm computed on a slice and the Hilbert space of $L^2$ holomorphic maps from $\cU$ to $\C^4$ are isometric. Therefore, in view of Remark \ref{rem_integkernel2}, we can extend the Bergman projection to the quaternionic setting.

\begin{teorema}Let $K_{\cU}$ be the classical Bergman kernel on $\cU$. We define $K_U$ as the extention of $K_{\cU}$ as a slice-regular function. Then
\begin{enumerate}
\item $K_U(q,s)=\overline{K_U(s,q)}$,
\item $K_U(\cdot, \cdot)$ is slice-regular in the first variable and slice-antiregular in the second one,
\item $K_U$ is a reproducing kernel, i.e.
$$f(q)=\int_{U\cap\C_v}K(q,s)f(s)d\mu$$
for every $v\in\sfera$ and every $q\in U$.
\end{enumerate}
\end{teorema}
\begin{proof} The listed properties are well-known for the classical Bergman kernel $K_{\cU}$, hence the first two follow from the fact that $K_U$ is the extension of $K_{\cU}$. The third one follows by Remark \ref{rem_integkernel2}, applying the reproducing kernel property of $K_{\cU}$ with the components of $F$.
\end{proof}

\section{Clifford Algebras}

As stated in \cite{GP1}, the definition of slice-regular functions via stem functions works in the same way for a real alternative algebra. However, all the computations we did in the beginning of this paper cannot be carried over verbatim to the general case.

Let us consider, for instance, the case of $\R_3$, the Clifford algebra on $3$ generators; through the standard basis
$$\{1,e_0,e_1,e_2,e_0e_1,e_0e_2,e_1e_2,e_0e_1e_2\}$$
we identify $\R_3$ with $\R^8$ as vector spaces. For $a\in\R_3$, we denote by $a_\ell$, $\ell=0,\ldots,7$ its components as an element of $\R^8$ and by $\|a\|$ its Euclidean norm.

We define the set of imaginary units as
$$S=\{u\in\R_3\ :\ u_0=0,\ \|u\|=1,\ u^2=-1\}\;.$$
Given a holomorphic function
$$F:\cU\to\C\otimes\R_3\cong\R^8$$
we obtain a slice-regular function $f:U\to\R_3$, where $U=\{x+uy\ :\ (x+\iota y)\in\cU,\ u\in S\}$.

Our object of study is now the set
$$Z=\{z\in\C^8\ :\ \exists u\in S \textrm{ s.t. } \Pi(z,u)=0\}$$
where $\Pi:\C^8\times S\to\R_3$ is defined in analogy with what we did in Section 2.

Some calculations show that
$$S=\{ u\in\R^8\ :\ u_0=u_7=0,\ u_1^2+\cdots+u_6^2=1,\ u_1u_6-u_2u_5+u_3u_4=0\}$$
and
$$Z=\{ z\in\C^8\ :\ z_0z_7-z_1z_6+z_2z_5-z_3z_4=0,\ z_0^2+\cdots+z_7^2=0\}\;.$$
Defining $\Phi:\C^8\to\C$ as $\Phi(z)=z_0^2+\cdots+z_7^2$, we note that, if $z\in Z$ and $w\in\C^8$, then $\Phi(z\star w)=\Phi(z)\Phi(w)$, where $\star$ is the product induced on $\C^8$ by the isomorphism with $\C\otimes\R_3$.

We can thus obtain all the considerations about the geometry of the zeros of slice-regular functions; however, as $Z$ is not an hypersurface anymore, we cannot replicate the results in the spirit of Rouché theorem without changes. For example, from the argument principle we do not obtain a way of counting zeros exactly, but only an upper bound on the number of zeros.

If we consider $\R_n$ with $n>3$, the codimension of the set $Z$ increases with $n$, making the direct computation of the equation impossible for the general case. However, this set and the set $S$ are linked to geometric and algebraic properties of $\R_n$; we will explore this connection in a future paper \cite{Mo}.

\end{document}